\documentclass[a4,12pt]{amsart}

\usepackage{amssymb,amsmath,amsfonts,amscd,bm,stmaryrd,verbatim}

 \theoremstyle{plain}
 \begingroup
 \newtheorem{thm}{Theorem}[section]
 \newtheorem{lem}[thm]{Lemma}

 \endgroup

 \theoremstyle{definition}
 \begingroup

 \endgroup

 \theoremstyle{remark}
 \begingroup
 
 \endgroup

 \numberwithin{equation}{section}

\def\Xint#1{\mathchoice
   {\XXint\displaystyle\textstyle{#1}}%
   {\XXint\textstyle\scriptstyle{#1}}%
   {\XXint\scriptstyle\scriptscriptstyle{#1}}%
   {\XXint\scriptscriptstyle\scriptscriptstyle{#1}}%
   \!\int}
\def\XXint#1#2#3{{\setbox0=\hbox{$#1{#2#3}{\int}$}
     \vcenter{\hbox{$#2#3$}}\kern-.5\wd0}}

\def\dashint{\Xint-}

\title[Partial regularity]{Partial regularity  for minimizers of  singular energy functionals, with application to  liquid crystal models}

\author{Lawrence C.\ Evans, Olivier Kneuss and  Hung  Tran}

\address{Evans and Kneuss: Department of Mathematics\\University of California, Berkeley \newline
\indent Tran: Department of Mathematics, University of Chicago}

\thanks{LCE is  supported in part by NSF Grants DMS-1001724 and DMS-1301661, OK is supported by
Swiss NSF Grant 143575, and HT is supported in part by NSF Grant DMS-1001724}

\begin{document}
\maketitle

\noindent {\bf Abstract.}
We study the partial regularity of minimizers for certain singular functionals
in the calculus of variations, motivated by Ball and Majumdar's recent
modification  \cite{BM1}
of the  Landau-de Gennes energy functional.

\begin{section}{Introduction}

\medskip

{\bf 1.1 A singular variational problem.}
In this paper we establish the partial regularity of minimizers $\mathbf{u} \in H^1(U;\Bbb R^k) $ for singular energy functionals having the form
\begin{equation}
\label{energy.functional.1}
I[\mathbf{v}] :=\int_U F(\mathbf v, D\mathbf{v})+  f(\mathbf{v})\, dx
\end{equation}
where $F$ is quasiconvex in the gradient variables and the convex function $f$ blows up to infinity at the boundary of  a given bounded open set $\Bbb K \subset \Bbb R^k$. As we will explain later in Section 5, this sort of energy functional arises in some recently proposed models in 
nematic liquid crystal theory.

\medskip

We assume hereafter that  $U \subset \mathbb R^n$ is bounded smooth domain
 and that $\Bbb K$ is a  bounded, open convex subset of $ \mathbb R^k$. Our assumptions are these:

\medskip

 {\bf (H1) Hypotheses on \boldmath $f$\unboldmath}:
The given function $f: \mathbb R^k \to [0,\infty]$
is nonnegative, convex and  smooth on $\Bbb K \subset \mathbb R^k$. We will  write $f = f(z) $.

We further require that
\begin{equation}
 \label{f.hypotheses}
 \begin{cases}
f(z) <\infty \qquad&\mbox{if }\,z\in  \Bbb K,\\
f(z) =\infty  &\mbox{if }\,z \in \mathbb R^k- \Bbb K
\end{cases}
\end{equation}
and
\begin{equation}
 \label{f.hypotheses.blowup}
f(z)\to \infty  \ \ \mbox{ as } \mbox{dist} (z,\partial \Bbb K) \to 0, z \in \Bbb K.
\end{equation}

\medskip

 {\bf (H2)  Hypotheses on \boldmath $F$\unboldmath }:   We assume $F:\mathbb R^k \times \mathbb M^{k \times n} \to \mathbb R$ is given, $\mathbb M^{k \times n}$ denoting the space of real, $k \times n$ matrices. We write $F=F(z,P)$.

We suppose as well that  $F$ is {\em uniformly strictly quasiconvex} in the $P$ variables. This means that  there exists a constant  $\gamma>0$ such that
\begin{equation}
\label{quasiconvex}
\int_V F(z,P) + \gamma |D\mathbf{w}|^2 \,dx \le \int_V F(z,P+D\mathbf{w}) \,dx,
\end{equation}
 for each   smooth bounded domain $V \subset \mathbb R^n$, each  $z \in \mathbb R^k$ and  $P \in \mathbb M^{k \times n}$, and all $\mathbf{w} \in C^1(V; \mathbb R^k)$ satisfying $\mathbf{w}=0$ on $\partial V$. The  physical significance  of quasiconvexity is discussed for instance
 in the foundational paper \cite{B} of Ball.

\medskip

We introduce the  further technical assumptions  that
\begin{equation}
\begin{cases}
 |D^2_{P} F(z,P)| \le C, \\
 \gamma |P|^2   \le F(z,P) + C,  \\
 |F(z,P)-F(\hat z,P)| \le C (1+|P|^2)|z- \hat z|,
\end{cases}
\end{equation}
for appropriate constants $C, \gamma>0$ and all  $z, \hat z \in \mathbb R^k, P \in  \mathbb M^{k \times n}$.

\medskip
{\bf (H3)  Hypothesis on admissible mappings}: We propose to minimize the functional
$I[\cdot]$ over the admissible class of functions
$$
\mathcal A : = \{ \mathbf{v} \in H^1(U;  \mathbb R^k) \  \mid
 \mathbf{v} =  \mathbf{g} \  \text{ on $\partial U$ in the trace sense} \},
$$
where the given smooth function $ \mathbf{g}: \partial U \to \mathbb R^k$ provides the boundary conditions. For this we need to assume
\begin{equation}
\text{there exists  $ \mathbf{u}^* \in \mathcal A$ with finite energy: $I[\mathbf{u}^*] < \infty$.}
\end{equation}

\medskip
Under the hypotheses (H1)-(H3), standard arguments in the calculus of variations prove the existence of a minimizer $\mathbf{u} \in \mathcal A$:
\begin{equation}
I[\mathbf{u}] = \min_{\mathbf{v} \in \mathcal A}I[\mathbf{v}] < \infty.
\end{equation}
The key question that we address in this paper is the regularity of $\mathbf{u}$.
Since $I[\mathbf{u}] < \infty$, we certainly have $ \mathbf{u} \in \Bbb K$ almost everywhere,
but conceivably $ \mathbf{u}(x)$ lies in $\mathbb R^k - \Bbb K$ for a dense set of points $x \in U$.

\medskip {\bf Remark:} Our hypothesis  that the second derivatives in $P$ of $F$ are bounded is
restrictive for quasiconvex integrands, as most polyconvex $F$ will not satisfy this. Our partial  regularity assertions are in fact valid under more general growth conditions, but to keep this paper at a reasonable length, we omit the proofs: see for instance \cite{Evans1}. \qed

\end{section}

\begin{section}{Partial regularity for a model problem} The  proof of partial regularity is
a fairly straightforward modification of standard, but  rather complicated, variational techniques (cf  \cite{EG1}), with particular
attention paid to the singular term involving the function $f$.

\medskip

To keep the presentation fairly simple, we devote  this section to a   simplified model where $F=F(P)$ depends only on the gradient.
We therefore consider now  the  energy functional
\begin{equation}
\label{sim.form}
I[\mathbf v] = \int_U F(D\mathbf v) +  f(\mathbf v) \,dx,
\end{equation}
and hereafter assume that $\mathbf u \in \mathcal A$ is a minimizer.

\medskip

{\bf 2.1 Linear approximation.}
Given a ball $B(x_0,r) \subset U$, we define the quantity
\begin{equation}
\label{def.E}
E(x_0,r) :=r^{1/2}+\dashint_{B(x_0,r)} |D\mathbf{u}-(D\mathbf{u})_{x_0,r}|^2\,dx,
\end{equation}
 which measures the averaged $L^2$-deviation of $D\mathbf{u}$ over the ball from
its average value
$$
(D\mathbf{u})_{x_0,r} :=\dashint_{B(x_0,r)} D\mathbf{u}\,dx.
$$
We later use also the similar notation
$$
(\mathbf{u})_{x_0,r} :=\dashint_{B(x_0,r)} \mathbf{u}\,dx.
$$
In the above formulas, the slash through the integral sign means the average over the ball
$B(x_0,r)$.

\medskip The following  assertion is the key to $C^1$ partial regularity:

\begin{thm} \label{decay}
For each $L>0$, there exists a constant $C=C(L)$ with the property  that for each $\tau \in (0,\frac{1}{8})$ there exists $\epsilon = \epsilon(L,\tau)>0$ such that
\begin{equation}
|(\mathbf{u})_{x_0,r}|, |f((\mathbf{u})_{x_0,r})|, |(D\mathbf{u})_{x_0,r}| \le L
\end{equation}
and
\begin{equation}
E(x_0,r) \le \epsilon
\end{equation}
imply
\begin{equation}
E(x_0,\tau r) \le C \tau^{1/2} E(x_0,r)
\end{equation}
 for each ball $B(x_0,r) \subset U$.

\end{thm}

\begin{proof} 1.  We  argue by contradiction.
Should the  Theorem be false,
there would exist balls $\{B(x_m,r_m)\}_{m=1}^\infty \subset U$ such that
\begin{equation}
\label{de1}
|(\mathbf{u})_{x_m,r_m}|, |f((\mathbf{u})_{x_m,r_m})|, |(D\mathbf{u})_{x_m,r_m}|  \le L,
\end{equation}
and
\begin{equation}
\label{de1.5}
E(x_m, r_m)=: \lambda_m^2 \to 0,
\end{equation}
but
\begin{equation}
\label{contraction}
E(x_m, \tau r_m) > C \tau^{1/2}  \lambda_m^2,
\end{equation}
for a constant $C$ we will select later.

\medskip

2.  We have from \eqref{def.E} and \eqref{de1.5} that
\begin{equation}
\label{estimate.radius}
r_m^{1/2} \le \lambda_m^2.
\end{equation}
Also
\begin{equation}
\label{de2}
\lambda_m^{-2} \dashint_{B(x_m, r_m)} |D\mathbf{u}-(D\mathbf{u})_{x_m, r_m}|^2\,dx \le 1.
\end{equation}
We combine \eqref{de2} with \eqref{de1}, to discover
$$
\dashint_{B(x_m, r_m)} |D\mathbf{u}|^2\, dx \le C.
$$

\medskip

Put $a_m:=(\mathbf{u})_{x_m,r_m},\ A_m := (D\mathbf{u})_{x_m,r_m}$, and introduce the rescaled functions
$$
\mathbf{v}_m(z)= \frac{ \mathbf{u}(x_m+r_mz)-a_m-r_mA_m z} {\lambda_m r_m}
$$
 for $z \in B:=B(0,1)$.
Then
$$D\mathbf{v}_m(z)= \frac {D\mathbf{u}(x_m+r_mz)-A_m}{\lambda_m},
 $$ and
$$
(\mathbf{v}_m)_B=(D\mathbf{v}_m)_B=0.
$$
Observe also  that
$$
\dashint_B |D\mathbf{v}_m(z)|^2\,dz = \lambda_m^{-2} \dashint_{B(x_m,r_m)} |D\mathbf{u} -(D\mathbf{u})_{x_m,r_m}|^2\,dx  \le 1.
$$
Since $(\mathbf{v}_m)_B=0$, Poincar\'e's inequality then provides the bound
$$
\dashint_B |\mathbf{v}_m|^2 \, dz \le C.
$$

Passing if necessary  to a subsequence and relabelling, we may suppose that
\begin{equation}
\label{de3}
\begin{cases}
\mathbf{v}_m \to \mathbf{v}\,\mbox{ strongly in }\,L^2(B; \Bbb R^k),\\
D\mathbf{v}_m \rightharpoonup D\mathbf{v}\,\mbox{ weakly in }\,L^2(B; \Bbb M^{k \times n}).
\end{cases}
\end{equation}
Also since $|a_m|, |A_m| \le L$,
we may assume also that
$$
a_m \to a, \quad A_m \to A.
$$

\medskip

3. We hereafter  write
$$
\Bbb K_\delta := \{q \in \Bbb K|\,\mbox{dist}(q,\partial \Bbb K) >\delta\}.
$$
for  small $\delta>0$. According to \eqref{de1}, we have
$$
|f(a)|=\lim_{m\to \infty} |f(a_m)| \le L.
$$
It consequently follows from  \eqref{f.hypotheses.blowup} that there exists  $\epsilon_0>0$ such that $a \in \Bbb K_{2\epsilon_0}$;
and hence there exists a sufficiently large index $M$
such that $a_m\in  \Bbb K_{\epsilon_0}$ for $m>M$.

\medskip

Recalling that  $\mathbf{u}$ is a minimizer and  rescaling $B(x_m,r_m)$ to the unit ball $B$, we see that
\begin{multline}
\label{de4}
\dashint_B  F(A_m+\lambda_m D\mathbf{v}_m) + f(a_m+r_mA_m z + \lambda_m r_m \mathbf{v}_m)  \,dz \\  \le \dashint_B F(A_m+\lambda_m D \tilde {\mathbf{v}}_m)  + f(a_m+r_mA_m z + \lambda_m r_m \tilde{ \mathbf{v}}_m) \,dz,
\end{multline}
 provided  $\tilde{\mathbf{v}}_m \in H^1(B; \Bbb R^k)$
 and $\tilde {\mathbf{v}}_m = \mathbf{v}_m$ on $\partial B$.
Then 
 \begin{equation}\label{de5}
 \dashint_B DF(A_m) \cdot D \mathbf{v}_m \,dz = \dashint_B DF(A_m)\cdot D \tilde { \mathbf{v}}_m\, dz.
 \end{equation}
It follows that $ \mathbf{v}_m$ is a minimizer of
$$
I_r^m[ \mathbf{w}]=\int_{B(0,r)} F_m( D \mathbf{w}) +  \frac{1}{\lambda_m^2}f(a_m+r_mA_m z + \lambda_m r_m  \mathbf{w}) \, dz,
$$
subject to its boundary conditions, for  the rescaled energy density
\[
F_m(P) :=\frac{F(A_m+\lambda_m P)-F(A_m)-\lambda_m DF(A_m) : P}{\lambda_m^2}
\]
and $r \in (0,1]$. In other words,
\begin{equation}
\label{variational.principle.m}
I_r^m[ \mathbf{v}_m] \le I_r^m[ \mathbf{w}]
\end{equation}  for any $ \mathbf{w} \in H^1(B(0,r); \Bbb R^k)$ such that $ \mathbf{w}= \mathbf{v}_m$ in $\partial B(0,r)$.

 \medskip

4. To streamline the presentation, we sequester various intricate calculations into the proofs of
two technical lemmas that follow this main proof.

\medskip

According to the following Lemma \ref{ELsys}
the limit $\mathbf{v}$ is a weak solution of the constant coefficient,  uniformly elliptic system
\eqref{elliptic.sys}.
 Standard regularity theory (cf. for instance \cite{Gia1}) implies then that
  $\mathbf{v}$ is smooth.  In particular we have the bound
\[
\max_{B(0, \frac 1 2)} |D^2\mathbf{v}| \le C \dashint_B| D\mathbf{v}|^2 \le C.
\]
Consequently
\[
\dashint_{B(0, \tau)} |D\mathbf{v}-(D\mathbf{v})_{0,\tau}|^2\,dx \le C_1 \tau^2
\]
for some constant $C_1= C_1(L)$.

\medskip

However, rescaling the inequalities
\eqref{contraction} and using \eqref{estimate.radius} gives
\[\dashint_{B(0, \tau)} |D\mathbf{v}_m-(D\mathbf{v}_m)_{0,\tau}|^2\,dx \ge  (C-1) \tau^{1/2}.
\]
But owing to the following Lemma \ref{strong.conv}, we have the strong convergence  $\mathbf{v}_m \to  \mathbf{v}$
in $H^1(B(0, \tau); \Bbb R^k)$. This leads to the desired contradiction, provided we take $C = C_1+2$.
\end{proof}

\medskip The previous proof invoked the following two technical  lemmas.

\begin{lem}
\label{strong.conv}
$D \mathbf{v}_m$ converges strongly to $D \mathbf{v}$ in $L^2_{loc}(B;\Bbb M^{k \times n})$.
\end{lem}
\begin{proof}
1. We firstly define a Radon measure $\mu_m$ on $B= B(0,1)$ by
$$
\mu_m(A)=\int_A |D \mathbf{v}_m|^2+|D \mathbf{v}|^2 \,dx,
$$
for any Borel set $A \subseteq B $. Since $\{ \mu_m(B) \}_{m=1}^\infty$ is bounded,  we may assume, passing if necessary to a subsequence,  that there exists a Radon measure $\mu$ on $B$ such that
$$
\mu_m \rightharpoonup \mu\,\mbox{ weakly in the sense of measures}.
$$
We then also have $\mu(B) <\infty$; whence
\begin{equation}
\label{mzero}
\mu(\partial B(0,r))=0
\end{equation}
 for  all but at most countably many $r \in (0,1]$.
 Select any $r \in (0,1)$ such that \eqref{mzero} holds.

 \medskip
 2. For $R \in (r,1)$, let $\xi$ be a smooth cutoff function satisfying
 $$
 \begin{cases}
0\le \xi \le 1;  \xi \equiv 1 \,\mbox{ on }\,B(0,r); \\
  \xi \equiv 0\, \mbox{ on }\,\mathbb R^n - B(0,R); \ \
 |D\xi| \le \frac{C}{R-r}.
 \end{cases}
 $$
 Define $\bm{\phi}_m=(\phi_m^1, \dots,\phi_m^k )$, where
 \begin{equation} \label{phi}
 \phi_m^j(x)=
 \begin{cases}
 \frac{1}{r_m},\quad &\mbox{if }\,v^j(x) \ge \frac{1}{r_m}\\
-  \frac{1}{r_m},\quad &\mbox{if }\,v^j(x) \le -\frac{1}{r_m}\\
  v^j(x),\quad&\mbox{if }\, -\frac{1}{r_m} <v^j(x) < \frac{1}{r_m}
 \end{cases}
 \end{equation}
 Then $$ r_m |\bm{\phi}_m | \le C $$ and so
  $  \lambda_m r_m |\bm {\phi}_m| \le C \lambda_m \to 0$ uniformly.
 Since $a_m\in \Bbb K_{\epsilon_0}$, it follows that for $m$ large enough
 \begin{equation}
 \label{control}
a_m +r_mA_m z,\, a_m +r_mA_m z+\lambda_m r_m \bm{\phi}_m \in K_{\epsilon_0/2}.
 \end{equation}
for all $z \in B$.
 \medskip

 Observe also that
$$
\int_B |\bm\phi_m - \mathbf{v}|^2 \, dz \le \sum_{j=1}^k \int_{ \{ r_m | v^j| >1 \} }| \mathbf{v}|^2 \, dz \to 0
$$
and
$$
\int_B |D\bm\phi_m - D\mathbf{v}|^2 \, dz \le \sum_{j=1}^k \int_{ \{ r_m | v^j| >1 \} }| D\mathbf{v}|^2 \, dz \to 0.
$$
Hence \begin{equation}
 \label{strong.convergence.H.one}
\bm\phi_m\rightarrow \mathbf{v} \text{ in $H^1(B;\mathbb{R}^k)$}.
 \end{equation}
  \medskip
 3.
Put $$\tilde {\mathbf{v}}_m := \xi \bm{\phi}_m + (1-\xi) \mathbf{v}_m.$$ Then
 $
 D\tilde {\mathbf{v}}_m = \xi D\bm{\phi}_m + (1-\xi) D{\mathbf{v}}_m + (\bm{\phi}_m-\mathbf{v}_m)D\xi.
 $

 \medskip

 We now assert that
 \begin{equation}
  \label{lim}
 \limsup_{m\to \infty} (I_r^m[\mathbf{v}_m]-I_r^m[\bm{\phi}_m]) \le 0.
 \end{equation}

To see this, note that $I_R^m[\mathbf{v}_m] \le I_R^m[\tilde {\mathbf{v}}_m]$, according
to \eqref{variational.principle.m}.  Consequently,
 \begin{align*}
&\qquad  0  \ge I_R^m[\mathbf{v}_m] - I_R^m[\tilde {\mathbf{v}}_m]   \\
& \qquad =I_r^m[\mathbf{v}_m]-I_r^m[\bm{\phi}_m]  +\int_{B(0,R) - B(0,r)} F_m(D\mathbf{v}_m)-F_m(D\tilde {\mathbf{v}}_m)\,dz \\
& \qquad   \qquad \qquad \qquad   +\frac{1}{\lambda_m^2} \int_{B(0,R) - B(0,r)} f(a_m+r_mA_m z + \lambda_m r_m \mathbf{v}_m)  \notag  \\
& \qquad   \qquad \qquad \qquad \qquad \qquad \qquad \qquad - f(a_m+r_mA_m z +  \lambda_m r_m \tilde {\mathbf{v}}_m)\,dz;
 \end{align*}
and so
\begin{equation}
\label{two.terms.to.estimate}
 \begin{aligned}
 I_r^m[\mathbf{v}_m]-I_r^m[\bm{\phi}_m] & \le \int_{B(0,R) - B(0,r)} F_m(D\tilde {\mathbf{v}}_m)-F_m(D \mathbf{v}_m)\,dz \\
&    +\frac{1}{\lambda_m^2} \int_{B(0,R) - B(0,r)} f(a_m+r_mA_m z + \lambda_m r_m \tilde {\mathbf{v}}_m) \\
&\qquad \qquad  \qquad \qquad- f(a_m+r_mA_m z + \lambda_m r_m  \mathbf{v}_m) \,dz.
 \end{aligned}
 \end{equation}

 Now
  \begin{align*}
&\int_{B(0,R) - B(0,r)} F_m(D\mathbf{v}_m)\, dz \\
&=\frac{1}{\lambda_m^2} \int_{B(0,R) - B(0,r)} F(A_m+\lambda_m D\mathbf{v}_m) - F(A_m)-\lambda_mDF(A_m) \cdot D\mathbf{v}_m\,dz\\
& = \int_{B(0,R) - B(0,r)} \int_0^1 \int_0^1 s(D\mathbf{v}_m)^T \cdot D^2F(A_m + st \lambda_m D\mathbf{v}_m)  D\mathbf{v}_m \,dt\,ds\,dz. \notag
 \end{align*}
Likewise
  \begin{align*}
&\frac{1}{\lambda_m^2} \int_{B(0,R) - B(0,r)} F(A_m+\lambda_m D\tilde {\mathbf{v}}_m) - F(A_m)-\lambda_mDF(A_m) \cdot D\tilde {\mathbf{v}}_m\,dz\\
=& \int_{B(0,R) - B(0,r)} \int_0^1 \int_0^1 s(D\tilde {\mathbf{v}}_m)^T \cdot D^2F(A_m + st \lambda_m D\tilde {\mathbf{v}}_m)  D\tilde {\mathbf{v}}_m \,dt\,ds\,dz. \notag
 \end{align*}
 Combining the foregoing, we deduce that
 \begin{equation}
 \label{estimate.F_m}
 \begin{aligned}
& \int_{B(0,R) - B(0,r)} F_m(D\tilde {\mathbf{v}}_m)-F_m(D \mathbf{v}_m)\,dz \\
& \qquad \qquad  \le C \int_{B(0,R) - B(0,r)} |D\mathbf{v}_m|^2+|D\bm{\phi}_m|^2 + |D\xi|^2 |\bm{\phi}_m-\mathbf{v}_m|^2\,dz \\
&  \qquad \qquad \le\,C \mu(\overline{B(0,R) - B(0,r)}) + o(1)
 \end{aligned}
 \end{equation}
as $m\to \infty$, where we have used \eqref{de3} and \eqref{strong.convergence.H.one}.

\medskip

We next consider the terms involving $f$, taking particular care since $f$ blows up at $\partial \Bbb K$. The convexity of $f$ and \eqref{control} yield
 \begin{align}
 \label{de9}
 &\frac{1}{\lambda_m^2} [f(a_m+r_mA_mz+\lambda_m r_m\tilde {\mathbf{v}}_m)-f(a_m+r_mA_mz+\lambda_m r_m \mathbf{v}_m)] \notag \\
& \le  \frac{1}{\lambda_m^2} [\xi f(a_m+r_mA_mz+\lambda_m r_m \bm{\phi}_m)+(1-\xi)f(a_m+r_mA_mz+\lambda_m r_m \mathbf{v}_m) \notag \\
& \qquad\qquad\qquad\qquad\qquad-f(a_m+r_mA_mz+\lambda_m r_m \mathbf{v}_m)] \notag\\
 & = \frac{1}{\lambda_m^2} \xi [f(a_m+r_mA_mz+\lambda_m r_m \bm{\phi}_m)-f(a_m+r_mA_mz+\lambda_m r_m \mathbf{v}_m)] \notag\\
& \le  \frac{1}{\lambda_m^2} \xi [f(a_m+r_mA_mz)+\lambda_m r_m Df(a_m+r_mA_mz) \cdot \bm{\phi}_m+C |\lambda_m r_m \bm{\phi}_m|^2-\notag\\
 &\qquad\qquad\qquad\qquad\qquad-f(a_m+r_mA_mz)-\lambda_m r_m Df(a_m+r_mA_mz)\cdot \mathbf{v}_m)] \notag\\
& \le C  \xi ( \frac{r_m}{\lambda_m}|\bm{\phi}_m-\mathbf{v}_m| +  r_m^2 |\bm{\phi}_m|^2),\notag
 \end{align}
  $C$ depending upon $\max_{z \in \overline{\Bbb K_{\epsilon_0/2}}} (|Df(z)|+|D^2f(z)|)$.

 \medskip

 According to \eqref{estimate.radius}, \eqref{de2}, \eqref{de3}, and \eqref{strong.convergence.H.one}, we see that
  $$
   \frac{r_m}{\lambda_m}|\bm{\phi}_m-\mathbf{v}_m| \to 0
 $$ in $L^1(B)$. Furthermore,
$ r_m^2 |\bm{\phi}_m|^2 \le C$ and $r_m^2 |\bm{\phi}_m|^2 \to 0$ almost everywhere.
  Hence the Dominated Convergence Theorem implies
 \begin{align*}
\limsup_{m\to \infty}\frac{1}{\lambda_m^2} \int_{B(0,R) - B(0,r)} &\left[f(a_m+r_mA_m z + \lambda_m r_m \tilde {\mathbf{v}}_m) \right.\\
-& \left.f(a_m+r_mA_m z + \lambda_m r_m  \mathbf{v}_m)\right] \,dz\le 0.
 \end{align*}

 \medskip

We recall the  previous estimates \eqref{two.terms.to.estimate} and  \eqref{estimate.F_m},  to conclude that
 $$
 \limsup_{m\to\infty}( I_r^m[\mathbf{v}_m]-I_r^m[\bm{\phi}_m] )\le  C \mu(\overline{B(0,R) - B(0,r)})
 $$
 Letting $R \to r$ and remembering \eqref{mzero}, we obtain the assertion \eqref{lim}.

 \medskip
4.  Given $0<s<r$, we let $\rho$ be another smooth cutoff function such that
 $$
 \begin{cases}
 0\le \rho \le 1;\   \rho \equiv 1 \,\mbox{ on }\,B_s; \\
    \rho \equiv 0\,\mbox{ on }\,\mathbb R^n - B(0,r), \
 |D\rho| \le \frac{C}{r-s}.
 \end{cases}
 $$
Define $$\bm{\psi}_m=\rho(\mathbf{v}_m-\bm{\phi}_m)$$ and notice that $\bm{\phi}_m+\bm{\psi}_m=\rho \mathbf{v}_m +(1-\rho)\bm{\phi}_m$.

 \medskip

Then
 \begin{align*}
 I_r^m [\mathbf{v}_m] - I_r^m[\bm{\phi}_m]
&  =( I_r^m [\mathbf{v}_m] - I_r^m[\bm{\phi}_m+\bm{\psi}_m]) + ( I_r^m [\bm{\phi}_m+\bm{\psi}_m] - I_r^m[\bm{\phi}_m]) \\
&    =: R_1+R_2.
 \end{align*}
Proceeding as above,
 \begin{equation}
 \label{de11}
 -R_1=I_r^m[\rho \mathbf{v}_m+(1-\rho)\bm{\phi}_m]- I_r^m [\mathbf{v}_m] \le C \mu(\overline{B(0,r) - B(0,s)}) +o(1),
 \end{equation}
as $m\to \infty$.  The term $R_2$ can be written as
 \begin{align}
 R_2=& \int_{B(0,r)} F_m(D\bm{\psi}_m) \,dz+ \int_{B(0,r)} F_m(D\bm{\phi}_m + D\bm{\psi}_m)-F_m(D\bm{\phi}_m)-F_m(D\bm{\psi}_m) \,dz \notag\\
 &+\frac{1}{\lambda_m^2}
 \int_{B(0,r)} f(a_m+r_mA_m z + \lambda_m r_m(\bm \phi_m+\bm\psi_m)) \\
 & \qquad  \qquad  \qquad  \qquad  \qquad  \qquad \qquad  \qquad    -f(a_m+r_mA_m z + \lambda_m r_m\bm\phi_m)\,dz \notag\\
 =: &\,S_1+S_2+S_3.\notag
 \end{align}

 \medskip

Now the convexity of $f$ implies
\begin{align*}
 -S_3  & = \frac{1}{\lambda_m^2}
 \int_{B(0,r)} f(a_m+r_mA_m z + \lambda_m r_m\bm\phi_m)- f(a_m+r_mA_m z + \lambda_m r_m(\bm \phi_m+\bm\psi_m))\,dz \\
 &  \le \frac{1}{\lambda_m^2}\int_{B(0,r)}  Df(a_m+r_mA_m z + \lambda_m r_m\bm\phi_m) \cdot
 (- \lambda_m r_m \bm\psi_m )\,dz  \\
& \le \frac{Cr_m}{\lambda_m}\int_{B(0,r)}  |\bm\psi_m |\,dz
 \le \frac{Cr_m}{\lambda_m}\int_{B(0,r)}  |\mathbf{v}_m| +|\mathbf{v}|\,dz  \\
& \le \frac{Cr_m}{\lambda_m} \le C\lambda_m^3
  = o(1)
 \end{align*}
as $m\to \infty$, according to \eqref{estimate.radius} and \eqref{control}.

\medskip

The uniform strict quasiconvexity of $F$ yields
 \begin{align*}
 S_1&=\int_{B(0,r)} F_m(D\bm\psi_m)\,dz \\
 & = \frac{1}{\lambda_m^2} \int_{B(0,r)} F(A_m+\lambda_m D\bm\psi_m)-F(A_m)-\lambda_m DF(A_m)\cdot D\bm\psi_m\,dz\\
 &= \frac{1}{\lambda_m^2} \int_{B(0,r)} F(A_m+\lambda_m D\bm\psi_m)-F(A_m)\,dz \ge \gamma \int_{B(0,r)} |D\bm\psi_m|^2\,dz.\notag
 \end{align*}
Since
\begin{align}
F_m(P+Q)-F_m(P)-F_m(Q)&=(\int_0^1(DF_m(P+tQ)-DF_m(tQ))\,dt)\cdot Q\notag\\
&=P^T  (\int_0^1 \int_0^1 D^2F_m(sP+tQ)\,ds\,dt) Q,\notag
\end{align}
we obtain
$$
 S_2=\int_{B(0,r)} (D\bm\phi_m)^T  G_m  D\bm\psi_m\,dz,
$$
for
\begin{multline*}
G_m=\int_0^1 \int_0^1 D^2F_m(sD\bm\phi_m+tD\bm\psi_m)\,ds\,dt \\
=\int_0^1 \int_0^1 D^2F(A_m+s\lambda_mD\bm\phi_m+t\lambda_mD\bm\psi_m)\,ds\,dt.
\end{multline*}
Since $\lambda_mD\bm\phi_m, \lambda_mD\bm\psi_m\rightarrow \mathbf{0}$ strongly in $L^1$ we deduce that (up to a subsequence)
$\lambda_mD\bm\phi_m(x), \lambda_mD\bm\psi_m(x)\rightarrow \mathbf{0}$ a.e. Hence, recalling that $D\bm\phi_m \to D\mathbf{v}$ strongly in $L^2$ and that $D^2F$ is bounded and continuous, we find that, using the Dominated Convergence Theorem,
$$(D\bm\phi_m)^T  G_m\rightarrow (D\mathbf{v})^T  F(A)\quad \text{ strongly in }L^2.$$
As $D\bm\psi_m \rightharpoonup \mathbf{0}$
weakly in $L^2$, we therefore get that
$$
S_2=o(1),\quad\mbox{as }\,m\to \infty.
$$

\medskip

Combining the foregoing estimates on $R_1, S_1, S_2, S_3$, we eventually find that
$$
\limsup_{m \to \infty} \int_{B(0,r)} |D\bm\psi_m|^2 \,dz\le C\mu(\overline{B(0,r)- B(0,s)}).
$$
Hence for any $0<s<r$,
\begin{equation}
\label{estimate.with.phi.m}
\limsup_{m \to \infty} \int_{B(0,s)} |D\mathbf{v}_m-D\bm\phi_m|^2 \,dz\le C\mu(\overline{B(0,r)- B(0,s)}).
\end{equation}

\medskip

Hence \eqref{estimate.with.phi.m} and  \eqref{strong.convergence.H.one} imply
$$
\limsup_{m \to \infty} \int_{B(0,s)} |D\mathbf{v}_m-D\mathbf{v}|^2 \,dz\le C\mu(\overline{B(0,r)- B(0,s)}).
$$
Our sending $s\to r$ completes the proof.
\end{proof}

\medskip
We need one further assertion,  that $\mathbf{v}$ solves a linear elliptic system (and consequently is smooth.) We again have to take care, as $f$ is singular:

\begin{lem}
\label{ELsys}
The function $\mathbf{v}$ satisfies the integral identity
\begin{equation}
\label{sys}
\int_{B(0,r)} (D\mathbf{w})^T D^2F(A) D\mathbf{v}\,dz=0
\end{equation}
for all  $\mathbf{w} \in H_0^1(B(0,r); \Bbb R^k)$.

\medskip

Consequently, $\mathbf{v}$ is a weak solution of the constant coefficient elliptic system
\begin{equation}
\label{elliptic.sys}
\text{\rm div} \left( D^2F(A) D\mathbf{v} \right)=0.
\end{equation}
\end{lem}
\begin{proof} 1. First we show that
for any $\bm\varphi \in C^\infty(B; \Bbb R^k)$,
\begin{equation}
\label{var}
\int_{B(0,r)} (D\tilde {\mathbf{v}})^T D^2F(A) D\tilde {\mathbf{v}} \,dz\ge \int_{B(0,r)} (D\mathbf{v})^T D^2F(A) D\mathbf{v}\,dz,
\end{equation}
for $\tilde {\mathbf{v}} = \rho \bm\varphi+(1-\rho)\mathbf{v}$, where $\rho$ is a cutoff function as in Lemma \ref{strong.conv}.

\medskip

To prove this, we set $\tilde {\mathbf{v}}_m :=\rho \bm\varphi + (1-\rho) \mathbf{v}_m$. According to \eqref{variational.principle.m},
$$I_r^m [\mathbf{v}_m] \le I_r^m[\tilde {\mathbf{v}}_m].
$$

 As before, the convexity of $f$ implies
 \begin{align}
 &\frac{1}{\lambda_m^2} [f(a_m+r_mA_mz+\lambda_m r_m\tilde {\mathbf{v}}_m)-f(a_m+r_mA_mz+\lambda_m r_m \mathbf{v}_m)] \notag \\
 & \le \frac{1}{\lambda_m^2} \rho [f(a_m+r_mA_mz+\lambda_m r_m  \bm\varphi)-f(a_m+r_mA_mz+\lambda_m r_m \mathbf{v}_m)] \notag\\
& \le  \frac{1}{\lambda_m^2} \rho [f(a_m+r_mA_mz)+\lambda_m r_m Df(a_m+r_mA_mz) \cdot
\bm\varphi+C |\lambda_m r_m \bm\varphi|^2-\notag\\
 &\qquad\qquad\qquad\qquad\qquad-f(a_m+r_mA_mz)-\lambda_m r_m Df(a_m+r_mA_mz)\cdot \mathbf{v}_m] \notag\\
& \le C   ( \frac{r_m}{\lambda_m}|\bm\varphi-\mathbf{v}_m| +  r_m^2 |\bm\varphi|^2).\notag
 \end{align}
Therefore
$$
\limsup_{m\to \infty} \frac{1}{\lambda_m^2} \int_{B(0,r)} f(a_m+r_mA_mz+\lambda_m r_m \tilde {\mathbf{v}}_m)-f(a_m+r_mA_mz+\lambda_m r_m  \mathbf{v}_m)\,dz \le 0.
$$

\medskip

Thus
$$
\int_{B(0,r)} F_m(D\mathbf{v}_m) \,dz\le \int_{B(0,r)} F_m(D\tilde {\mathbf{v}}_m)\,dz+o(1)
$$
as $m\to \infty$. Repeating the calculations before, the above inequality is equivalent to
\begin{align}
&\int_{B(0,r)} \int_0^1 \int_0^1 s (D\mathbf{v}_m)^T D^2F(A_m + ts\lambda_m D\mathbf{v}_m) D\mathbf{v}_m \,dt\,ds\,dz  \notag\\
\le &\int_{B(0,r)} \int_0^1 \int_0^1 s (D\tilde {\mathbf{v}}_m)^T D^2F(A_m + ts\lambda_m D\tilde {\mathbf{v}}_m) D\tilde {\mathbf{v}}_m \,dt\,ds\,dz+o(1).\notag
\end{align}
Lemma \ref{strong.conv} shows that $D\mathbf{v}_m \to D\mathbf{v}$ in $L^2_{loc}$. Letting $m\to \infty$,
we derive the inequality \eqref{var}.

\medskip

2. By  approximation  we see that \eqref{var} is still valid for $\tilde {\mathbf{v}} = \mathbf{v} + \lambda \mathbf{w}$ for $\mathbf{w} \in C_c^\infty(B(0,s))$ and $\lambda >0$.
 Hence
$$
\int_{B(0,r)} (D\mathbf{v}+\lambda D\mathbf{w})^T D^2 F(A) (D\mathbf{v}+\lambda D\mathbf{w})\,dz \ge \int_{B(0,r)} (D\mathbf{v})^T D^2F(A) D\mathbf{v}\,dz.
$$
We expand out the left hand side and cancel the terms that do not involve $\lambda$. Dividing  by $\lambda >0$ and then sending $\lambda \to 0$, we find that
$$
\int_{B(0,r)} (D\mathbf{w})^T D^2F(A) D\mathbf{v}\,dz \ge 0
$$
Replacing $ \mathbf{w}$ with $ -\mathbf{w}$, we get the reverse inequality, and so
 \eqref{sys} follows.
\end{proof}

\medskip

{\bf 2.2 Iteration}

We next recursively apply Theorem 2.1 on smaller and smaller concentric balls.

\begin{lem}\label{add}
Given  $L>0$, let $C_1= C(2L)$ is the constant from Theorem 2.1. Then for each  $\tau$ satisfying
\begin{equation}
0<\tau <\min \left (\frac{1}{8},\frac{1}{4C_1^2} \right),
\end{equation}
there exists $\eta = \eta(L,\tau)>0$ such that
\begin{equation}
|(\mathbf{u})_{x,r}|, |f((\mathbf{u})_{x,r})|, |(D\mathbf{u})_{x,r}|\le L
\end{equation}
and
\begin{equation}
E(x,r) \le \eta
\end{equation}
imply
\begin{equation}
E(x,\tau^l r) \le C_1\tau^{1/2} E(x,\tau^{l-1}r)  \qquad (l = 1, \dots)
\end{equation}
for each ball $B(x,r) \subset U$.
\end{lem}

\medskip

\begin{proof} 1. We first note from hypothesis (H1) that
for each $L>0$, there exists $\epsilon_1 = \epsilon_1(L)\in (0,1)$ such that
\begin{equation}
\label{prop.g}
\mbox{ if  $f(a)<L$ and $ |a-b| \le \epsilon_1$, then $ f(b) <2L$.}
\end{equation}

Let $\epsilon_2 =\epsilon(2L,\tau)$ be as in Theorem 2.1.
Define
$$
\eta =\min \left ( L^2,\epsilon_2,\left(\frac{\epsilon_1(1-\tau)\tau^n}{(1+L)C_2}\right)^{1/2}, \left (\tau^n L(1-C_1^{1/2}\tau^{1/4})\right )^2 \right ),
$$
the constant $C_2$ to be selected below.
\medskip

2. We assert next  that the following inequalities hold for all $l \ge 0$:
\begin{equation}\label{ind1}
|(\mathbf{u})_{x,\tau^l r}| \le 2L,
\end{equation}
\begin{equation}\label{ind2}
|f((\mathbf{u})_{x,\tau^l r})| \le 2L,
\end{equation}
\begin{equation}\label{ind3}
|(D\mathbf{u})_{x,\tau^l r}| \le 2L,
\end{equation}
%\begin{equation}\label{ind4}
%|(D\mathbf{u})_{x,\tau^{l+1} r}| \le 2L,
%\end{equation}
\begin{equation}\label{ind5}
E(x,\tau^l r) \le \eta \le \epsilon_2.
\end{equation}

\medskip The proof is by induction, the case
 $l=0$ being the hypothesis.
Assume next  that \eqref{ind1}-\eqref{ind5} are valid for $l=0,1,...,p-1$;
we will show that are also valid for $l=p$.

\medskip

{\bf Proof of \eqref{ind1}:} Poincar\'e's inequality implies for each $l \le p-1$ that
\begin{align}
\label{ind6}
|(\mathbf{u})_{x,\tau^{l+1}r}-(\mathbf{u})_{x,\tau^l r}|
& \le \dashint_{B(x,\tau^{l+1} r)} |\mathbf{u}-(\mathbf{u})_{x,\tau^l r}|\, dy\\
 &\le \left(\dashint_{B(x,\tau^{l+1} r)} |\mathbf{u}-(\mathbf{u})_{x, \tau^l r}|^2 \,dy \right)^{1/2}\notag\\
 &\le  \frac{1}{\tau^n} \left (\dashint_{B(x,\tau^l r)} |\mathbf{u}-(\mathbf{u})_{x, \tau^l r}|^2 \,dy \right)^{1/2}\notag \\
& \le \frac{C \tau^l r} {\tau^n} \left (\dashint_{B(x,\tau^l r)} |D\mathbf{u}|^2 \,dy \right )^{1/2}.\notag
\end{align}
The induction hypothesis then gives
\begin{align*}
\dashint_{B(x,\tau^l r)} |D\mathbf{u}|^2\, dy & \le 2 \dashint_{B(x,\tau^l r)}
 |D\mathbf{u}-(D\mathbf{u})_{x,\tau^l r}|^2 +|(D\mathbf{u})_{x,\tau^l r}|^2\,dy  \\
 & \le 2(\eta + (2L)^2)\le 10L^2.
\end{align*}

Plugging  into \eqref{ind6}, we  find that
\begin{equation}
\label{ind7}
|(\mathbf{u})_{x,\tau^{l+1}r}-(\mathbf{u})_{x,\tau^l r}| \le \frac{C_2Lr \tau^l}{\tau^n}.
\end{equation}
Thus
\begin{align*}
|(\mathbf{u})_{x,\tau^p r}| & \le |(\mathbf{u})_{x,r}|+\sum_{l=0}^{p-1} |(\mathbf{u})_{x,\tau^{l+1}r}-(\mathbf{u})_{x,\tau^l r}| \notag\\
& \le  L + \frac{C_2Lr}{\tau^n} \sum_{l=0}^{p-1} \tau^l \le L \left(1+\frac{C_2 r}{\tau^n(1-\tau)} \right) \le 2L,
\end{align*}
since
\[
r \le E(x,r)^2 \le \eta^2 \le \frac{(1-\tau)\tau^n}{C_2}.
\]

\medskip

{\bf Proof of \eqref{ind2}:} Using \eqref{ind7}, we see that
\begin{equation}
\label{ind8}
|(\mathbf{u})_{x,\tau^p r}-(\mathbf{u})_{x,r}| \le \sum_{l=0}^{p-1} |(\mathbf{u})_{x,\tau^{l+1}r}-(\mathbf{u})_{x,\tau^l r}| \le \frac{C_2 Lr}{\tau^n(1-\tau)} \le \epsilon_1,
\end{equation}
 the last inequality holding since
$$
r \le E(x,r)^2 \le \eta^2 \le \frac{\epsilon_1(1-\tau)\tau^n}{(1+L)C_2} \le \frac{\epsilon_1(1-\tau)\tau^n}{LC_2}.
$$
So \eqref{ind8} and \eqref{prop.g} imply $|f((\mathbf{u})_{x,\tau^{p}r})| \le 2L$.

%\medskip
%{\bf  Proof of \eqref{ind3}:} The induction hypothesis immediately implies \eqref{ind3}.

\medskip
{ \bf Proof of \eqref{ind3}:} For $l \le p-1$, using the induction hypothesis and Lemma \ref{decay}  we have
\begin{align*}
|(D\mathbf{u})_{x,\tau^{l+1}r}-(D\mathbf{u})_{x,\tau^l r}|
& \le  \frac{1}{\tau^n} \left(\dashint_{B(x,\tau^l r)} |D\mathbf{u}-(D\mathbf{u})_{x, \tau^l r}|^2\, dy \right)^{1/2} \\
& \le \frac{1}{\tau^n}\,E(x,\tau^l r)^{1/2} \\
&  \le \frac{1}{\tau^n} [C_1^{1/2} \tau^{1/4}]^l \eta^{1/2}.\notag
\end{align*}
Therefore
\begin{align*}
|(D\mathbf{u})_{x,\tau^{p+1}r}| & \le |(D\mathbf{u})_{x,r}|+\sum_{l=0}^{p-1} |(D\mathbf{u})_{x,\tau^{l+1}r}-(D\mathbf{u})_{x,\tau^l r}| \notag\\
&  \le L + \frac{1}{\tau^n}\eta^{1/2} \sum_{l=0}^{p-1} [C_1^{1/2} \tau^{1/4}]^l \\
& \le L+\frac{\eta^{1/2}}{\tau^n(1-C_1^{1/2} \tau^{1/4})} \le 2L, \notag
\end{align*}
since $\eta^{1/2} \le \tau^n L  (1-C_1^{1/2} \tau^{1/4})$.

\medskip
{\bf  Proof of \eqref{ind5}:}  the induction hypothesis and Lemma \ref{decay} yield
$$
E(x,\tau^p r) \le (C_1 \tau^{1/2})^p \eta \le \eta.
$$

\medskip
3. Finally combining \eqref{ind1}-\eqref{ind5} and Lemma \ref{decay} one immediately obtains the lemma.
\end{proof}

\medskip
{\bf 2.3 Partial regularity.}
We are at last ready to state and prove our main assertion of partial regularity:
\begin{thm}\label{main}
There exists an open set $U_0 \subset U$ such that
\[
|U- U_0|=0
\] and, for every $\alpha \in (0,1/4)$,
\[
\mathbf{u} \in C^{1,\alpha}(U_0; \Bbb R^k).
\]
\end{thm}

\medskip

\begin{proof}
 1. Set
\begin{align*}
U_0 : & =\left \{x \in U \mid \lim_{r\to 0} (\mathbf{u})_{x,r}=\mathbf{u}(x),\, \lim_{r\to 0} (D\mathbf{u})_{x,r}=D\mathbf{u}(x), |\mathbf{u}(x)|<\infty,  \right. \\
&\qquad  \qquad  \left.  |D\mathbf{u}(x)|<\infty , \,f(\mathbf{u}(x))<\infty ,\,\lim_{r\to 0} \dashint_{B(x,r)} |D\mathbf{u}-(D\mathbf{u})_{x,r}|^2\,dy=0 \right \}.
\end{align*}
Then $|U - U_0|=0$, since $\mathbf{u} \in H^1(U)$ and $\int_U f(\mathbf{u})\,dx<\infty$.

\medskip
2. We assert that $U_0$ is open and $D\mathbf{u} \in C^\alpha(U_0)$ for $0<\alpha<1/4.$

\medskip

For each $x \in U_0$, there exist $L=L(x)$ and $R=R(x) \in (0, \text{dist}(x,\partial U))$ such that
\begin{align}
\label{mt1}
\begin{cases}
 |(\mathbf{u})_{x,s}|, |(D\mathbf{u})_{x,s}| <L \  \mbox{ for all }\,0<s<R. \\
 |f((\mathbf{u})_{x,s})| < L  \   \mbox{ for all }\,0<s<R.
 \end{cases}
\end{align}
Fix $\alpha \in (0,1/4)$. Take $\tau \in \left(0,\min \left(\frac{1}{8},\frac{1}{4C_1^2}\right)\right)$ such that
$$
C_1 \tau^{1/2-2\alpha} <1.
$$

We then can choose  $0<r<R$ so  small enough  that
\begin{equation}
\label{mt2}
E(x,r)=r^{1/2}+\dashint_{B(x,r)} |D\mathbf{u}-(D\mathbf{u})_{x,r}|^2 \,dy<\eta(L,\tau),
\end{equation}
where $\eta$ has been constructed in Lemma \ref{add}.

In summary, \eqref{mt1} and \eqref{mt2} imply
\begin{equation}\label{mt3}
|(\mathbf{u})_{x,r}|, |f((\mathbf{u})_{x,r})| , |(D\mathbf{u})_{x,r}|<L, \text{ and } E(x,r)<\eta(L,\tau).
\end{equation}

Moreover,  the following mappings
$$
x \mapsto (\mathbf{u})_{x,r}, f((\mathbf{u})_{x,r}) , (D\mathbf{u})_{x,r},  E(x,r)
$$
are continuous. Hence \eqref{mt3} holds for $z \in B(x,s)$ for some $s>0$.

Applying Lemma \ref{add}, for any $z \in B(x,s)$
$$
E(z,\tau^l r) \le (c(2L)\tau^{1/2})^l \eta(L,\tau) \le (\tau^l r)^{2\alpha} \eta_1(L,\tau,r)
$$
for $l = 1, \dots$, where $\eta_1(L,\tau,r)=\eta(L,\tau) r^{-2\alpha}$.
The previous estimate now implies (cf. for instance \cite{Gia1}) that $D\mathbf{u} \in C^\alpha$ near $x.$
This immediately shows that $\mathbf{u}\in C^{\alpha}(U_0)$ and that $U_0$ is open.
\end{proof}

\end{section}

\begin{section}{Partial regularity for the general problem}

In this section we return to the general functional
 $$I(\mathbf{v})=\int_U F(\mathbf{v},D\mathbf{v})+f(\mathbf{v}) \, dx $$
 and outline the requisite modifications in the previous proof of partial regularity.

\begin{thm}\label{thm.reg.C.1.alpha.with.param} Let $\mathbf{u}$ be a minimizer of $I[ \, \cdot \,]$.

Then there exists an open set $U_0\subset U$ such that $$|U-U_0|=0$$ and, for each $\alpha\in (0,1/4),$  $$\mathbf{u}\in C^{1,\alpha}(U_0; \Bbb R^k).$$
\end{thm}
We start by an elementary lemma which will allow us to reduce the problem to the model problem ($F=F(P)$).

\begin{lem}\label{lem.red.case.with.no.param} Let $a_m\in \mathbb{R}^k$ and $A_m\in \mathbb{M}^{k\times n}$ be bounded and let $\mathbf{w}_m\in H^1(B;\mathbb{R}^k)$ be bounded.
Let $r_m,\lambda_m>0$ be such that $r_m\leq \lambda_m^4\rightarrow 0$ as $m\rightarrow \infty$.

Then, upon passing if necessary to a subsequence, we have
\begin{align*}&\int_B|F(a_m,A_m+\lambda_mD\mathbf{w}_m)-
F(a_m+r_mA_mz+\lambda_mr_m\mathbf{w}_m,A_m+\lambda_mD\mathbf{w}_m)|dz\\
&\qquad \qquad \qquad \qquad\qquad \qquad\qquad \qquad\qquad \qquad\qquad \qquad\qquad \qquad=o(\lambda_m^2).
\end{align*}
\end{lem}

\begin{proof} 1. Using our hypothesis  (H2) we deduce that
\begin{align*}& \int_B|F(a_m,A_m+\lambda_mD\mathbf{w}_m)-
F(a_m+r_mA_mz+\lambda_mr_m\mathbf{w}_m,A_m+\lambda_mD\mathbf{w}_m)|dz\\
&\leq C\int_B(1+|A_m+\lambda_mD\mathbf{w}_m|^2)\,|r_mA_mz+\lambda_mr_m\mathbf{w}_m| \, dz\\
&\leq  C\lambda_m^2\int_B \left[r_m/\lambda_m^2+r_m|\mathbf{w}_m|/\lambda_m\right] \, dz\\
&\qquad + C\lambda_m^2\int_B |D\mathbf{w}_m|^2\min\{1,r_m+r_m\lambda_m|\mathbf{w}_m|\} \, dz \\
=&: R_1+ R_2.
\end{align*}

It is elementary to see that $R_1=o(\lambda_m^2)$.

\medskip
2.  Since $\mathbf{w}_m$ is bounded in $H^1$, we may assume that there exists $\mathbf{w}\in H^1$ such that
$D\mathbf{w}_m\rightharpoonup D\mathbf{w}$ weakly in $L^2$.
Moreover since $\|r_m\lambda_m\mathbf{w}_m\|_{L^1(B)}\rightarrow 0$ we deduce that
$$\lim_{m\rightarrow \infty}r_m+r_m\lambda_m|\mathbf{w}_m|(x)=0\quad\text{ a.e.}$$
Fix $\epsilon>0$. According to Ergoroff's Theorem there exists a measurable set $A_{\epsilon}\subset B$ such that
$$|B-A_{\epsilon}|<\epsilon\quad\text{and}\quad r_m+r_m\lambda_m|\mathbf{w}_m|\rightarrow 0\text{ uniformly in $A_{\epsilon}$.}$$
We hence have
\begin{align*}
R_2&\leq C_2\lambda_m^2\int_{A_{\epsilon}}|D\mathbf{w}_m|^2(r_m+r_m\lambda_m|\mathbf{w}_m|)\, dz+ C_2\lambda_m^2\int_{B-A_{\epsilon}}|D\mathbf{w}_m|^2 \, dz\\
&\leq C_2\lambda_m^2\|D\mathbf{w}_m\|_{L^2{(A_{\epsilon})}}\|\left[r_m+r_m\lambda_m|\mathbf{w}_m|\right]\|_{L^{\infty}{(A_{\epsilon})}}+ C_2\lambda_m^2\int_{B-A_{\epsilon}}|D\mathbf{w}_m|^2 \, dz\\
&=o(\lambda_m^2)+C_2\lambda_m^2\int_{B-A_{\epsilon}}|D\mathbf{w}_m|^2 \, dz\\
&=o(\lambda_m^2)+C_2\lambda_m^2\int_{B-A_{\epsilon}}|D\mathbf{w}|^2 \, dz.
\end{align*}
Letting $\epsilon \to 0$ we immediately obtain the result.
\end{proof}

We now sketch the proof of Theorem \ref{thm.reg.C.1.alpha.with.param}. As already said the proof is almost identical to the one of Theorem \ref{main}.

\begin{proof} We first claim that Theorem \ref{decay}, Lemma \ref{strong.conv} and Lemma \ref{ELsys} still hold.
Define $a_m,A_m,v_m,r_m,\lambda_m,E(x,r)$ as in Section 2. Then (compare with \eqref{de4})
\begin{align*}
&\dashint_B  F(a_m,A_m+\lambda_m D\mathbf{v}_m) + f(a_m+r_mA_m z + \lambda_m r_m \mathbf{v}_m)  \,dz \\
+&\dashint_B   F(a_m+r_mA_m z + \lambda_m r_m \mathbf{v}_m,A_m+\lambda_m D\mathbf{v}_m)-F(a_m,A_m+\lambda_m D\mathbf{v}_m)\,dz  \\
  \le& \dashint_B F(a_m,A_m+\lambda_m D \tilde {\mathbf{v}}_m)  + f(a_m+r_mA_m z + \lambda_m r_m \tilde{ \mathbf{v}}_m) \,dz\\
  +&\dashint_B   F(a_m+r_mA_m z + \lambda_m r_m \tilde {\mathbf{v}}_m,A_m+\lambda_m D\tilde {\mathbf{v}}_m)-F(a_m,A_m+\lambda_m D\tilde {\mathbf{v}}_m)\,dz,
\end{align*}
 provided  $\tilde{\mathbf{v}}_m \in H^1(B)$
 and $\tilde {\mathbf{v}}_m = \mathbf{v}_m$ on $\partial B$.
It follows that $ \mathbf{v}_m$ is a minimizer of
\begin{multline*}
I_r^m[ \mathbf{w}]=\int_{B(0,r)} F_m( D \mathbf{w}) +  \frac{1}{\lambda_m^2}f(a_m+r_mA_m z + \lambda_m r_m  \mathbf{w}) \, dz\\
+\frac{1}{\lambda_m^2}\int_{B(0,r)}  F(a_m+r_mA_m z + \lambda_m r_m \mathbf{w},A_m+\lambda_m D\mathbf{w})-F(a_m,A_m+\lambda_m D\mathbf{w}) \, dz,
\end{multline*}
subject to its boundary conditions, for
\[
F_m(P) :=\frac{F(a_m,A_m+\lambda_m P)-F(a_m,A_m)-\lambda_m DF(a_m,A_m) \cdot P}{\lambda_m^2}
\]
and $r \in (0,1]$. In other words,
\begin{equation*}
I_r^m[ \mathbf{v}_m] \le I_r^m[ \mathbf{w}]
\end{equation*}  for any $ \mathbf{w} \in H^1(B(0,r))$ such that $ \mathbf{w}= \mathbf{v}_m$ in $\partial B(0,r)$.
By Lemma \ref{lem.red.case.with.no.param} we have that
\begin{align*}&\frac{1}{\lambda_m^2}\int_{B(0,r)}  F(a_m+r_mA_m z + \lambda_m r_m \mathbf{w}_m,A_m+\lambda_m D\mathbf{w}_m)-F(a_m,A_m+\lambda_m D\mathbf{w}_m)\, dz\\
&=o(1),
\end{align*}
as long as $\mathbf{w}_m$ is bounded in $H^1(B(0,r);\mathbb{R}^k).$

\medskip
We now proceed exactly as in Section 2 to obtain the claim. We have in effect reduced the problem to the case $F=F(P)$. Finally the end of the proof  is exactly the same as  in Section 2.
\end{proof}
\end{section}

\begin{section}{Improved estimates for convex $F$}
This section is devoted to the study of improved partial regularity for minimizers when $F$ is uniformly convex and only depends on the gradient, meaning that there exists a positive constant $\gamma$ such that
\begin{equation}
R^TD^2F(P)R \ge \gamma |R|^2
\end{equation}
for all matrices $P,R \in \mathbb M^{k \times n}$.

\medskip

{\bf 4.1 Second derivative estimates.} We first show that the second derivatives of our
minimizer exist and are locally square-integrable. This is a standard assertion in the calculus of variations when the singular term $f$ is absent:
see for instance Giaquinta \cite{Gia1} or \cite[Section 8.3.1]{EPDE}.

\begin{thm}
\label{h2reg}
Suppose in addition to the hypotheses of Section 2 that $F$ is uniformly convex. Then
\[
\mathbf{u} \in H^{2}_{loc}(U).
\]
\end{thm}

\begin{proof}
1. Fix any open set $V \subset\subset U$ and then select a cutoff function $\xi$ satisfying
$$
\begin{cases}
0 \le \xi \le 1, \ \xi=1\,\mbox{ on }\,V \\  \xi=0\,\mbox{ near } \partial U.
\end{cases}
$$
For $|h|>0$ small, let
$$
\mathbf{v}=D_k^{-h}(\xi^2 D_k^h \mathbf{u}),
$$
where $D_k^h \mathbf u$ denotes the difference quotient
$$
D_k^h \mathbf{u}(x) = \frac{\mathbf{u}(x+he_k)-\mathbf{u}(x)}{h} \quad (h \in \mathbb R,\,h \ne 0)
$$ and $e_k$ is the unit vector in the $x_k$  direction.
The explicit form of $\mathbf v$ is
\[
\frac{\xi^2(x) \mathbf{u}(x+he_k)+\xi^2(x-he_k) \mathbf{u}(x-he_k)-(\xi^2(x) +\xi^2(x-he_k)) \mathbf{u}(x)}{h^2}.
\]

\medskip

2. There exists an open set  $W$ such that $V\subset \subset   W  \subset \subset U$ and $\mbox{spt} (\mathbf v) \subset  W$ for $|h|>0$ small enough.
For  small $t= t(h) >0$, we have
$$
 1-t\frac{\xi^2(x)}{h^2}- t\frac{\xi^2(x-he_k)}{h^2} \ge 0.
$$
The convexity of $f$ consequently implies that
\begin{align}
\label{convexity.uplustvee}
\qquad f(\mathbf{u}+t\mathbf{v}) & =f \left( t\frac{\xi^2(x)}{h^2} \mathbf{u}(x+he_k) + t\frac{\xi^2(x-he_k)}{h^2} \mathbf{u}(x-he_k) \right. \notag \\
&\qquad \qquad \qquad \qquad + \left.   \left(1-t\frac{\xi^2(x)}{h^2}- t\frac{\xi^2(x-he_k)}{h^2}\right)\mathbf{u}(x) \right) \notag \\
& \le  t\frac{\xi^2(x)}{h^2} f(\mathbf{u}(x+he_k))+t\frac{\xi^2(x-he_k)}{h^2} f(\mathbf{u}(x-he_k))
 \\
&\qquad \qquad \qquad \qquad + \left(1-t\frac{\xi^2(x)}{h^2}- t\frac{\xi^2(x-he_k)}{h^2}\right)f(\mathbf{u}(x)). \notag
\end{align}

\medskip
3. We note next that
\begin{equation}
\label{sign.integrals.with.f}
 \int_{U} f(\mathbf{u}  +t\mathbf{v}) \, dx \le \int_U f(\mathbf{u})  \,dx .
\end{equation}
To confirm this,  observe from \eqref{convexity.uplustvee} that
\begin{align*}
   \int_{U} \frac{ f(\mathbf{u}  +t\mathbf{v})-f(\mathbf{u})}{t} \,dx & \le   \int_{U}
  \frac{\xi^2(x)}{h^2}( f(\mathbf{u}(x+he_k))-f(\mathbf{u}(x))) \, dx \\
  & \quad  +    \int_{U} \frac{\xi^2(x-he_k)}{h^2} (f(\mathbf{u}(x-he_k))- f(\mathbf{u}(x))) \, dx \\
  &=  \int_{U}
  \frac{\xi^2(x)}{h^2}( f(\mathbf{u}(x+he_k))-f(\mathbf{u}(x))) \, dx \\
  &\quad +    \int_{U} \frac{\xi^2(y)}{h^2} (f(\mathbf{u}(y))- f(\mathbf{u}(y + he_k))) \, dy \\
  &=0.
\end{align*}

\medskip

4.
Since $\mathbf{u}$ is a minimizer, we have for small $t>0$ that
\begin{align*}
0  &\le \int_U F(D\mathbf{u} +t D\mathbf{v}) - F(D\mathbf{u}) \,dx + \int_U f(\mathbf{u} +t \mathbf{v}) - f(\mathbf{u}) \,dx \\
&\le \int_U F(D\mathbf{u} +t D\mathbf{v}) - F(D\mathbf{u}) \,dx,
\end{align*}
according to \eqref{sign.integrals.with.f}. Divide by $t$ and send $t \to 0^+$:
\begin{equation}
\label{bett3}
0 \le   \int_{U}  DF(D\mathbf{u}) : D\mathbf{v}  \,dx.
\end{equation}

Recalling  the definition of $\mathbf{v}$, we see that
\begin{align*}
0 &\le   \int_{U} DF(D\mathbf{u}) : D\mathbf{v} \,dx
 =   \int_{U} F_{p_i^\alpha}(D\mathbf{u}) [D_k^{-h}(\xi^2 D_k^h u^\alpha)]_{x_i}\,dx\\
&= -  \int_{U}D_k^h(F_{p_i^\alpha}(D\mathbf{u})) [(\xi^2 D_k^h u^\alpha)]_{x_i}  \,dx.
\end{align*}
Following now a standard argument (as for instance in  \cite[Section 8.3.1]{EPDE}),  we use the uniform convexity of $F$
to bound the term
$$
  \int_{V} |D^h D{\bold u}|^2 \, dx
$$
independently of $h$. This implies that $D{\bold u}$ belongs to $H^1_{loc}$.
\end{proof}

Note that in this  proof we carefully avoided confronting the possibly very singular term  $Df(\mathbf{u})$. In particular,  we do not know that $Df(\mathbf{u})$ is integrable.

 \medskip

{\bf 4.2   Rate of blow-up of \boldmath$f$. \unboldmath}
If we know more about  the speed of blow-up of $f$ near the boundary of $\Bbb \mathbb{K}$, then  Theorem \ref{main} can be improved:

\begin{thm}\label{hausdorff1}
Assume that  $F$ is uniformly convex and there exists a constant $\gamma>0$ such that
\begin{equation}\label{fastdiv}
f(z) \ge \frac{\gamma}{{\rm dist}(z,\partial \mathbb{K})^2}, \qquad (z \in \mathbb{K}).
\end{equation}
Then
\begin{equation}
\mathcal{H}^{n-2 + \epsilon}(U- U_0)=0
\end{equation}
for each $\epsilon >0$, $\mathcal{H}^s$ denoting $s$-dimensional Hausdorff measure.
\end{thm}

\begin{proof}
According to Theorem \ref{h2reg}, $\mathbf{u} \in H^2_{loc}(U)$.

\medskip

1. Set
$$
 U_1 :=\left \{x \in U|\, \lim_{r\to 0}\dashint_{B(x,r)} |D\mathbf{u}-(D\mathbf{u})_{x,r}|^2\,dy=0\right\}.
$$
Using Poincar\'e's inequality and since $D^2\mathbf{u}\in L^2_{loc}(U)$ we have (cf. for instance \cite[page 77]{EG2})
\begin{align*}\mathcal{H}^{n-2}(U-U_1)=&\mathcal{H}^{n-2}\left(\left \{x \in U|\,\limsup_{r\to 0}\dashint_{B(x,r)} |D\mathbf{u}-(D\mathbf{u})_{x,r}|^2\,dy>0\right\}\right)\\
\le &\mathcal{H}^{n-2}\left(\left \{x \in U|\,\limsup_{r\to 0}\frac{1}{r^{n-2}}\int_{B(x,r)} |D^2\mathbf{u}|^2\,dy>0\right\}\right)=0.\end{align*}

\medskip

2. Set $$
 U_2 :=\left \{x \in U|\,\lim_{r\to 0} (\mathbf{u})_{x,r}=\mathbf{u}(x),\,\lim_{r\to 0} (D\mathbf{u})_{x,r}=D\mathbf{u}(x),\, |\mathbf{u}(x)|<\infty,\,|D\mathbf{u}(x)|<\infty\right\}.
$$
We claim that $\mathcal{H}^{n-2+\epsilon}(U- U_2)=0$ for every $\epsilon>0$.
This follows since if a function  belongs to $H^1$, then the limit of its averages over balls converges to a finite limit (in fact to the function itself) except possibly for a set $E$ with capacity $\text {Cap}_2(E) = 0$ (\cite[page 160]{EG2}), and therefore
$\mathcal{H}^{n-2 + \epsilon}(E) = 0$ (\cite[page 156]{EG2}).

\medskip

3. Let
$$
\Lambda =\left\{x\in U|\,\limsup_{r\to 0} \frac{1}{r^{n-2}}\int_{B(x,r)} f(\mathbf u)\,dy>0\right\}.
$$
Since $f (\mathbf  u ) \in L^1(U)$, we have as before that $\mathcal{H}^{n-2}(\Lambda)=0$:
see for instance  \cite[page 77]{EG2}.
\medskip

3. Define next
$$
U_3 :=\{x\in U_1\cap U_2|\, f(\mathbf{u}(x))=\infty\}.
$$

We claim that
\begin{equation}
\label{u2.and.lambda}
U_3 \subseteq \Lambda.
\end{equation}
To see this, take any $x\in U_3\subset U_1\cap U_2$. By the definition of $U_1$ and $U_2$, there exists $r, L>0$ such that
\begin{equation}
\label{u1.prop}
|(D\mathbf u)_{x,s}| \le L, \
\dashint_{B(x,s)} |D\mathbf u - (D\mathbf u)_{x,s}|^2\,dy \le L
 \ \mbox{for all } \, s\le r.
\end{equation}
Then (cf. Lemma \ref{add})
$$
|(\mathbf u)_{x, \tau^{l+1}s} - (\mathbf u)_{x,\tau^l s}| \le C \tau^l s
$$
for $l = 0,1, \dots$; and hence
$$
|(\mathbf{u})_{x,s} - \mathbf u(x)| \le \sum_{l=0}^\infty |(\mathbf u)_{x, \tau^{l+1}s} - (\mathbf u)_{x,\tau^l s}|  \le  \frac{Cs}{1-\tau}.
$$
As $u(x) \in \partial \mathbb{K}$,  the assumption \eqref{fastdiv} implies
\[
f((\mathbf{u})_{x,s}) \ge \frac{\gamma}{{\rm dist}((\mathbf{u})_{x,s},\partial \mathbb{K})^2}  \ge \frac{\gamma}{|(\mathbf{u})_{x,s} - \mathbf u(x)|^2} \ge \frac{C}{s^2}.
\]
Jensen's inequality now implies
$$
\dashint_{B(x,s)} f(\mathbf u)\,dy \ge f((\mathbf{u})_{x,s}) \ge \frac{C}{s^2}.
$$
Thus for $s<r$ we have
$$
\frac{1}{s^{n-2}} \int_{B(x,s)} f(\mathbf u)\,dy \ge C>0;
$$
and consequently $x \in \Lambda$.

\medskip

4. Observe  finally that  $U_0=(U_1\cap U_2)- U_3$ and $U- U_0 = (U- U_1) \cup (U-U_2)\cup U_3$. Hence $\mathcal{H}^{n-2 + \epsilon}(U- U_0)=0$ for each $\epsilon>0$.
\end{proof}

\medskip

As we will discuss in Section 4, Ball-Majumdar \cite{BM1} have introduced  certain liquid crystal models for which $f$ exhibits a logarithmic divergence near $\partial \Bbb K$: this is much weaker than \eqref{fastdiv}. We propose therefore to
extend the previous proof to handle this case, and for motivation look at the following model
problem:

\medskip

{\bf Example.}
Assume $\Bbb K=B(0,1)$ and there exists $r \in (1/2,1)$ such that
$$
f(z)=-\log(1-|z|), \qquad \mbox{for }\, r<|z|<1.
$$
Then
$$
f_{z^\alpha}(z)=\frac{z^\alpha}{|z|(1-|z|)},
$$
and
$$
f_{z^\alpha z^\beta}(z)=\frac{\delta_{\alpha \beta}}{|z|(1-|z|)}
+\frac{z^\alpha z^\beta}{|z|^2 (1-|z|)^2} \left(2-\frac{1}{|z|}\right).
$$
Therefore
\begin{equation}\notag
f_{z^\alpha z^\beta}(z) y^\alpha y^\beta \ge
 \gamma\, |Df(z)\cdot y|^2 \qquad \mbox{for } \,r<|z|<1,
\end{equation}
where $\gamma :=2-\frac{1}{r}$.
 \qed

\medskip

Motivated by this example, we introduce the condition that
\begin{equation}
\label{growth.log}
C|y|^2+ y^T D^2f(z) y
\ge \gamma\, |Df(z)\cdot y|^2,\qquad  (z \in \Bbb K,\,y\in \mathbb R^k)
\end{equation}
for constants $C, \gamma >0$.

\begin{thm}
\label{hausdorff2}
{\rm (i)} Assume that $F$ is uniformly convex and $f$ satisfies \eqref{growth.log}, then
\[f(\mathbf u) \in H^1_{loc}(U).
\]

{\rm (ii)} Therefore
\[
\mathcal{H}^{n-2+\epsilon}(U- U_0)=0
\]
for each $\epsilon>0.$
\end{thm}

\begin{proof} 1. Let $f^\eta$
be  smooth, convex and everywhere defined on $\mathbb R^k$, such that
$$
0 \le f^\eta \le f, \quad  f^\eta \equiv f \text{ on $\Bbb K_\eta$.}
$$
Let $\mathbf{u}^\eta$ denote the unique minimizer of
$$
I^\eta[\mathbf{v}]=\int_U F( D\mathbf{v})+  f^\eta(\mathbf{v})\, dx
$$
over the admissible class of functions $\mathcal A$. The functions $\{ \mathbf{u}^\eta \}_{\eta >0}$ are uniformly bounded in $H^2(V)$ for each compactly contained subregion
$V \subset \subset U$ (see the proof of Theorem \ref{h2reg}). As a consequence we claim that
\begin{equation}
\label{convergence.u.epsilon}
\mathbf{u}^\eta \to \mathbf{u} \quad \text{in $H^1_{loc}$}
\end{equation}
as $\eta \to 0$.
We already know that
$$ \mathbf{u}^\eta \to \mathbf{w} \quad \text{in $H^1_{loc}$}$$ for some $\mathbf{w}\in H^1_{loc}.$
It remains to show that $\mathbf{w}=\mathbf{u}.$
Note that
\begin{equation}\label{min.u.eta}\int_U f^\eta(\mathbf{u}^\eta)+F(D\mathbf{u}^\eta)   \,dx   \le \int_U f^\eta(\mathbf{u})+F(D\mathbf{u})   \,dx \le \int_U f(\mathbf{u})+F(D\mathbf{u})   \,dx.\end{equation}
We can assume that (up to a subsequence)
\begin{equation}\label{conv.point.u.eta}
\mathbf u^\eta \to  \mathbf w,\, D\mathbf u^\eta \to  D\mathbf w \ \ \text{a.e. in U.}
\end{equation}
Using the properties of $f^{\eta}$ and $F$ we hence deduce
$$f^{\eta}(\mathbf{u}^\eta(x)) \to f(\mathbf{w}(x)), F(D\mathbf{u}^\eta(x)) \to F(D\mathbf{w}(x))\ \ \text{a.e. in U.}$$
Combining \eqref{min.u.eta}, \eqref{conv.point.u.eta}  and  Fatou's lemma we have
$$ \int_U f(\mathbf{w})+F(D\mathbf{w})   \,dx \le \int_U f(\mathbf{u})+F(D\mathbf{u})   \,dx.$$
Hence by uniqueness of the minimizer of $I$ we obtain $\mathbf{u}=\mathbf{w}.$

\medskip
2. Since $\mathbf{u}^\eta$ is a minimizer of $I^{\eta}$ we have
$$
 \int_U f^\eta(\mathbf{u}^\eta) - f^\eta(\mathbf{u}^\eta +t \mathbf{v}^\eta)   \,dx   \le \int_U F(D\mathbf{u}^\eta +t D\mathbf{v}^\eta) - F(D\mathbf{u}^\eta) \,dx
$$
where
$$
\mathbf{v}^\eta=D_k^{-h}(\xi^2 D_k^h \mathbf{u}^\eta)
$$
and the function $\xi$ is as  in the proof of Theorem \ref{h2reg}.
Dividing by $t>0$ and sending $t \to 0$, we deduce
$$
- \int_U Df^\eta(\mathbf{u}^\eta)\cdot  \mathbf{v}^\eta   \,dx   \le \int_U DF(D\mathbf{u}^\eta): D\mathbf{v}^\eta \,dx \le C,
$$
owing to the uniform $H^2_{loc}$ estimates on $\mathbf{u}^\eta$. Rewriting the left hand side, we obtain the bound
$$
\int_{U} \xi^2 D_k^h(f^\eta_{z^\alpha} (\mathbf u^\eta)) D_k^h u^{\eta,\alpha}\,dx \le C.
$$

 Since $f^\eta$ is convex, the integrand in the last expression is pointwise nonnegative. Invoking  Fatou's Lemma, we deduce that
$$
  \int_{U} \xi^2 f^\eta_{z^\alpha z^\beta}(\mathbf u^\eta) u^{\eta, \alpha}_{x_k} u^{\eta, \beta}_{x_k} \,dx \le \liminf_{h\to 0}   \int_{U} \xi^2 D_k^h(f^\eta_{z^\alpha} (\mathbf u^\eta)) D_k^h u^{\eta,\alpha}\,dx \le C.
$$
Since $\xi \equiv 1$ on $V$, it follows that
\begin{equation}
\label{bound.Dsquared.u.epsilon}
\int_V  f^\eta_{z^\alpha z^\beta}(\mathbf u^\eta) u^{\eta, \alpha}_{x_k} u^{\eta, \beta}_{x_k}  \,dx  \le C,
\end{equation}
the constant independent of $\eta$.

\medskip
3. Fix a small $\delta >0$ and let
$$
A_\delta : = \{ x \in U \mid \mathbf u \in  \Bbb K_\delta  \}.
$$
Then \eqref{bound.Dsquared.u.epsilon} implies for $0 < \eta < \delta$ that
$$
\int_{V \cap A_\delta } f^\eta_{z^\alpha z^\beta}(\mathbf u^\eta) u^{\eta, \alpha}_{x_k} u^{\eta, \beta}_{x_k}  \,dx  \le C.
$$
Recall that we may assume that
$$
\mathbf u^\eta \to  \mathbf u, D\mathbf u^\eta \to  D\mathbf u \ \ \text{a.e. in U}
$$
Hence
$$
f^\eta_{z^\alpha z^\beta}(\mathbf u^\eta) u^{\eta, \alpha}_{x_k} u^{\eta, \beta}_{x_k}  \to f_{z^\alpha z^\beta}(\mathbf u) u^{ \alpha}_{x_k} u^{ \beta}_{x_k}
 \ \ \text{a.e. in $V \cap A_\delta$}
$$
Since $f^\eta_{z^\alpha z^\beta}(\mathbf u^\eta) u^{\eta, \alpha}_{x_k} u^{\eta, \beta}_{x_k} \ge 0$, we may invoke Fatou's Lemma, to deduce that
$$
\int_{V \cap A_\delta }f_{z^\alpha z^\beta}(\mathbf u) u^{ \alpha}_{x_k} u^{ \beta}_{x_k}\, dx \le \liminf_{\eta \to 0}\int_{V \cap A_\delta } f^\eta_{z^\alpha z^\beta}(\mathbf u^\eta) u^{\eta, \alpha}_{x_k} u^{\eta, \beta}_{x_k}  \,dx  \le C.
$$
Next, let $\delta \to 0$:
$$
\int_{V }f_{z^\alpha z^\beta}(\mathbf u) u^{ \alpha}_{x_k} u^{ \beta}_{x_k}\, dx \le    C.
$$

 We combine this with our assumption \eqref{growth.log}, to discover

\begin{equation}
\label{Df.in.L2}
 \int_{V } |Df(\mathbf u) D \mathbf u|^2\,dx \le \frac{1}{\gamma} \int_{V}
 (D\mathbf u)^TD^2f(\mathbf u) D\mathbf u + C |D\mathbf u|^2 \,dx \le C.
\end{equation}

\medskip

4.  We claim next that $Df(\mathbf u) D \mathbf u$ is the weak gradient of $f(\mathbf u)$ in the sense of distributions.
To see this, take a large number $\eta$ and define the open, convex set
$$
\Bbb F_\eta := \{ z \in \Bbb R^k  \mid f(z) < \eta  \}.
$$
%Select $\eta$ so that $\partial \Bbb F_\eta$ is smooth. 
Let
$
\Phi_\eta
$
denote the projection of $\Bbb R^k$ onto the closure of ${\Bbb F_\eta}$. Then $\Phi_\eta$ is Lipschitz continuous, with Lipschitz  constant  equal to one. Next let
$$
\mathbf u_\eta := \Phi_\eta(\mathbf u).
$$
Then
$$
0 \le f(\mathbf u_\eta) \le f(\mathbf u)  \in L^1(V),
$$
and therefore the Dominated Convergence Theorem implies that
$$
f(\mathbf u_\eta) \to f(\mathbf u) \quad \text{in $L^1(V)$}.
$$

Since  $\Phi_\eta$ has Lipschitz  constant  equal to one and  since $f$ restricted to $\Bbb F_\eta$ is smooth, we see that
$$ f(\mathbf{u}_{\eta}) \in H^1(V).$$
Now if a function belongs to the Sobolev space $H^1$ its gradient vanishes almost everywhere
on each level set.
Consequently,
$$Df(\mathbf{u}_{\eta}) D\mathbf{u}_{\eta}  = D(f(\mathbf{u}_{\eta})) = 0 \ \  \text{a.e. on  $\{ f(\mathbf{u}_{\eta}) = \eta \} = \{\mathbf{u} \notin \Bbb F_\eta
  \} $};
$$
and hence
$$
Df(\mathbf{u}_{\eta}) D\mathbf{u}_{\eta} =
\begin{cases}  Df(\mathbf{u}) D\mathbf{u}
\ \ \text{a.e. on  $ \{\mathbf{u} \in \Bbb F_\eta
  \} $}   \\
0 \qquad \qquad \ \  \text{a.e. on  $ \{\mathbf{u} \notin \Bbb F_\eta
  \}. $}
\end{cases}
$$
It follows that
$$
|Df(\mathbf{u}_{\eta}) D\mathbf{u}_{\eta}|  \le |Df(\mathbf{u}) D\mathbf{u}| \in L^2(V),
$$
according to \eqref{Df.in.L2}. Since $Df(\mathbf{u}_{\eta})D\mathbf{u}_{\eta}\to Df(\mathbf{u})D\mathbf{u}$ almost everywhere, the Dominated Convergence Theorem now implies
that
 $$Df(\mathbf{u}_{\eta})D\mathbf{u}_{\eta}\to  Df(\mathbf{u})D\mathbf{u}\quad \text{in $L^2(V).$}$$

 Hence for any function $\phi \in C^\infty(V)$ with compact support and for any $k = 1, \dots , n$, we have
\begin{align*}
\int_V f(\mathbf u) \phi_{x_k} \, dx & = \lim_{\eta \to \infty}\int_V f(\mathbf u_\eta) \phi_{x_k} \, dx
\\ &=  -\lim_{\eta \to \infty}\int_V ( f(\mathbf u_\eta))_{x_k} \phi \, dx  \\
&= -\lim_{\eta \to \infty}\int_V f_{z^{\alpha}}(\mathbf{u}_{\eta})(u^{\alpha}_{\eta})_{x_k} \phi \, dx  \\
&=  -\int_V  f_{z^{\alpha}}(\mathbf{u})u^{\alpha}_{x_k}\phi \, dx.
\end{align*}
Therefore $ D(f(\mathbf u))$ exists in the sense of distributions and  equals
 $ Df(\mathbf u) D \mathbf u \in L^2(V)$.

Next, put $(f(\mathbf u))_V :=\dashint_V f(\mathbf u)\,dx < \infty$.
According to Poincar\'e's inequality,
$$
\int_V |f (\mathbf u) - (f(\mathbf u))_V |^2 \,dx \le C \int_V |D(f (\mathbf u))|^2\,dx \le C.
$$
Thus
$$
\int_V |f(\mathbf u)|^2\,dx \ \le C.
$$
Hence $f ( \mathbf u) \in H^1_{loc}(U);$ which in turn implies (see the proof of Theorem \ref{hausdorff1})
$$
\mathcal{H}^{n-2+\epsilon}\left(\{x\in U|\, f(\mathbf{u}(x))=\infty\}\right)=0,
$$
for each $\epsilon> 0.$
Now proceeding similarly as in the proof of Theorem \ref{hausdorff1} we get that
$\mathcal{H}^{n-2+\epsilon}(U- U_0)=0$ for each each $\epsilon> 0.$
\end{proof}
\end{section}

\section{Applications.}

{\bf 5.1 Variational models for liquid crystals.} Our partial regularity theorems are motivated by some new physical models for nematic liquid crystals.

\medskip

 The nematic phase of a liquid crystal is a phase for which  the molecules are free to flow but  still tend to align, so as to have  long range directional order locally. These long range directions are locally approximately parallel. Thus the molecule at a point  $x$ has a preferred direction $\mathbf{n}(x)$ belonging to the unit sphere $\mathbb S^2$,  but it can also move around a little bit in the other two directions.

 There are  two well-known  mathematical models
of nematic liquid crystals,  the {\em mean field approach} and
and {\em Landau-de Gennes theory}.

\medskip

{\bf 5.2 Mean field models.}
In the mean field approach, the  alignment
of the nematic molecules at each point $x$ in space is described by
a probability distribution function $\rho_x$ on the unit sphere
satisfying  $\rho_x(p)=\rho_x(-p)$,  to model  the head-to-tail symmetry.
The first moment of $\rho_x$ hence vanishes:
$$
\int_{\mathbb S^2} p \rho_x \,d \mathcal{H}^2  =0.
$$
A corresponding  macroscopic  order parameter
is
\begin{equation}\label{macro}
Q(x) =\int_{\mathbb S^2} ( p \otimes p-{\textstyle \frac{1}{3}}I)  \rho_x \,d \mathcal{H}^2 .
\end{equation}
Thus  $Q$
is a symmetric, traceless $3 \times 3$ matrix
whose eigenvalues $\lambda_i = \lambda_i(Q)$ are constrained
by the inequalities (see e.g. \cite{Maj1})
\begin{equation}
 \label{eigen}
-{\textstyle \frac{1}{3}} \le \lambda_i \le {\textstyle \frac{2}{3}} \ \  (i=1,2,3);\  \sum_{i=1}^3\lambda_i=0.
\end{equation}
Notice that \eqref{eigen} is  equivalent to
\begin{equation}
 \label{eigen1}
-{\textstyle \frac{1}{3}} | \xi |^2 \le(Q \xi , \xi ) \le {\textstyle \frac{2}{3}}| \xi |^2\,\mbox{ for all } \xi \in \mathbb R^3;\  \sum_{i=1}^3 Q_{ii}=0.
\end{equation}
Consequently the set $\Bbb K$ of symmetric, traceless matrices $Q$ satisfying   \eqref{eigen} is bounded, closed and convex.

\medskip

{\bf 5.3 Landau-de Gennes models.} The Landau-de Gennes theory also describes the state of a nematic liquid crystal by the macroscopic  order parameter $Q$. However, now $Q$  is only required to be symmetric, traceless $3\times 3$ matrix; and here is no requirement of {\it a priori} bounds on the eigenvalues of $Q$ like \eqref{eigen}. The corresponding Landau-de Gennes energy functional is
$$
I_{LG}[Q] = \int_U f_B(Q) + F(Q, DQ) \,dx,
$$
where $f_B$ is a thermotropic bulk potential. As noted  in Ball-Majumdar \cite{BM1}, the function $f_B$ usually has the form
$$
f_B(Q)=\frac{1}{2} a\ \mbox{tr}\, Q^2 +{\textstyle \frac{1}{3}} b\ \mbox{tr} \,Q^3+\frac{1}{4} c\ (\mbox{tr}\, Q^2)^2+ \cdots,
$$
the coefficients $a,b,c, \dots$ depending  upon the temperature $T$.
In particular, in this model there is no term that enforces the physical constraints \eqref{eigen} on the eigenvalues. The equilibrium and physically observable configurations correspond either to global or local minimizers of the Landau-de Gennes energy subject to the imposed boundary conditions. Majumdar  observed in \cite{Maj1} that there are cases that the equilibrium order parameters can take value outside ${\Bbb K}$  even for temperatures  $T$ quite close to the nematic-isotropic transition temperature.

\medskip

Ball and Majumdar in \cite{BM1} address this issue by defining a new bulk potential
$$
f_B(Q)=T \psi(Q)-\kappa |Q|^2
$$
Here
$$
\psi(Q) = \inf_{\rho \in A_Q}\int_{\mathbb S^2}  \rho \log \rho \, d\mathcal{H}^{2},
$$
where
$$
A_Q = \left\{\rho: {\mathbb S^2} \to \Bbb R \mid \rho \ge 0,  \int_{\mathbb S^2}  \rho \, d\mathcal{H}^{2} = 1,   Q =\int_{\mathbb S^2} ( p \otimes p-{\textstyle \frac{1}{3}}I) \rho \, d\mathcal{H}^{2}\right\}.
$$
The
$f$ is convex and blows up a $ \partial \Bbb K$ (that is,  whenever the eigenvalues approach the limiting values of either $-{\textstyle \frac{1}{3}}$ or ${\textstyle \frac{2}{3}}$ in \eqref{eigen}). Ball and Majumdar  also showed that $f$ exhibits a logarithmic divergence as the eigenvalues approach either $-{\textstyle \frac{1}{3}}$ or ${\textstyle \frac{2}{3}}$.

\medskip

{\bf Example.}
In the case of modified Landau-de Gennes model,
we set $n=3$, $k=5$, and
we define a linear mapping $Q \mapsto q$
from the set of traceless, symmetric matrices
$\mathbb S^{3}$ to $\mathbb R^5$ as follows:
$$
q=(Q_{ij})_{i \ge j,\, i+j<6}.
$$
It is clear that this linear mapping is an isomorphism.
Then 
$$
\mathbf u(x) = (Q_{ij}(x))_{i \ge j,\, i+j<6},
$$
and the bounded open convex set $\Bbb K \subset \mathbb R^5$ is
$$
\Bbb K=\{q \in \mathbb R^5 \mid ~-{\textstyle \frac{1}{3}} <\lambda_i(Q) <{\textstyle \frac{2}{3}},\,\mbox{for } i=1,2,3\}.
$$ \qed

\bigskip
\bibliographystyle{alpha}

\end{document}